\documentclass[11pt]{article}
\usepackage[english]{babel}
\usepackage{amssymb,amsmath,amsthm}
\newtheorem{theorem}{Theorem}[section]
\newtheorem{lemma}[theorem]{Lemma}
\newtheorem{proposition}[theorem]{Proposition}
\newtheorem{corollary}[theorem]{Corollary}

\newtheorem{question}[theorem]{Question}

{\theoremstyle{definition}}
{\theoremstyle{definition}}
{\theoremstyle{definition}}
{\theoremstyle{definition}}

\newtheorem*{thmA}{Theorem~A}
\newtheorem*{thmB}{Theorem~B}

\numberwithin{equation}{section}

\def\C{{\mathbb C}}
\def\A{{\mathbb A}}
\def\B{{\mathbb B}}
\def\N{{\mathbb N}}
\def\Z{{\mathbb Z}}
\def\R{{\mathbb R}}

\def\K{{\mathbb K}}
\def\M{{\bf M}}
\def\T{{\bf T}}

\def\uu{{\cal U}}

\def\epsilon{\varepsilon}
\def\kappa{\varkappa}
\def\phi{\varphi}
\def\leq{\leqslant}
\def\geq{\geqslant}
\def\slim{\mathop{\hbox{$\overline{\hbox{\rm lim}}$}}\limits}
\def\ilim{\mathop{\hbox{$\underline{\hbox{\rm lim}}$}}\limits}
\def\supp{\hbox{\tt supp}\,}

\def\ker{\hbox{\tt ker}\,}

\def\re{\hbox{\tt Re}\,}

\def\uu{{\mathcal U}}

\title{On orbits of truncated convolution operators}

\author{Stanislav Shkarin}

\date{}

\begin{document}

\maketitle

\begin{abstract} We prove that a semigroup generated by a finitely
many truncated convolution operators on $L^p[0,1]$ with $1\leq
p<\infty$ is non-supercyclic. On the other hand, there is a
truncated convolution operator, which possesses irregular vectors.
\end{abstract}



\section{Introduction \label{s1}}\rm

Throughout the article all vector spaces are assumed to be over the
field $\K$ being either the field $\C$ of complex numbers or the
field $\R$ of real numbers, $\Z$ is the set of integers, $\Z_+$ is
the set of non-negative integers, $\R_+$ is the set of non-negative
real numbers and $\N$ is the set of positive integers. Symbol $L(X)$
stands for the space of continuous linear operators on a topological
vector space $X$ and $X^*$ is the space of continuous linear
functionals on $X$. A family ${\cal F}=\{F_a:a\in A\}$ of continuous
maps from a topological space $X$ to a topological space $Y$ is
called {\it universal} if there is $x\in X$ for which the orbit
$O({\cal F},x)=\{F_ax:a\in A\}$ is dense in $Y$. Such an $x$ is
called a {\it universal element} for $\cal F$. We use the symbol
$\uu({\cal F})$ to denote the set of universal elements for $\cal
F$. If $X$ is a topological vector space, and ${\cal F}$ is a
commutative subsemigroup of $L(X)$, then we call ${\cal F}$
hypercyclic if ${\cal F}$ is universal and the members of $\uu({\cal
F})$ are called {\it hypercyclic vectors} for ${\cal F}$. We say
that ${\cal F}$ is {\it supercyclic} if the semigroup ${\cal
F}_p=\{zT:z\in\K,\ T\in{\cal F}\}$ is hypercyclic and the
hypercyclic vectors for ${\cal F}_p$ are called {\it supercyclic
vectors} for ${\cal F}$. An orbit $O({\cal F}_p,x)$ will be called a
{\it projective orbit} of ${\cal F}$.  We refer to a tuple
$\T=(T_1,\dots,T_n)$ of commuting continuous linear operators on $X$
as {\it hypercyclic} (respectively, {\it supercyclic}) if the
semigroup generated by $\T$ is hypercyclic (respectively,
supercyclic). The concept of hypercyclic tuples of operators was
introduced and studied by Feldman \cite{feld}. In the case $n=1$, it
becomes the conventional hypercyclicity (or supercyclicity), which
has been widely studied, see the book \cite{bama-book} and
references therein.

Gallardo and Montes \cite{gm}, answering a question of Salas, proved
that the Volterra operator
\begin{equation}\label{volt}
Vf(x)=\int_0^x f(t)\,dt,
\end{equation}
acting on $L_p[0,1]$ for $1\leq p<\infty$, is non-supercyclic. This
lead to a quest of finding supercyclcic or even hypercyclic
operators as close as possible to the Volterra operator. In
\cite{ms1} it is observed that $L^2[0,1]$ admits a norm $\|\cdot\|$
defining a weaker topology such that $V$ is $\|\cdot\|$-continuous
and the continuous extension of $V$ to the completion of
$(L^2[0,1],\|\cdot\|)$ is hypercyclic. The mainstream of the above
quest dealt with searching of hypercyclic or supercyclic operators
commuting with $V$.

Truncated convolution operators form an important class of operators
commuting with $V$. Let $C_0[0,1]$ be the Banach space of continuous
functions $f:[0,1]\to\C$ satisfying $f(0)=0$ and carrying the
$\sup$-norm and let ${\bf M}$ be the space of finite
$\sigma$-additive $\K$-valued Borel measures $\mu$ on $[0,1)$. For
$\mu\in{\bf M}$, we consider the operator $C_\mu\in L(C_0[0,1])$
acting according to the formula
$$
C_\mu f(x)=\int_0^1 f_x(t)\,d\mu,\ \ \text{where $f_x(t)=f(x-t)$ if
$t\leq x$ and $f_x(t)=0$ if $t>x$.}
$$
In other words, $C_\mu f$ is the restriction to $[0,1]$ of the
convolution of $f$ and $\mu$. According to the well-known properties
of convolutions, $\|C_\mu f\|_p\leq \|\mu\|\|f\|_p$ for every $f\in
C_0[0,1]$, where $\|\mu\|$ is the full variation of $\mu$ and
$\|f\|_p$ is the norm of $f$ in $L^p[0,1]$ for $1\leq p\leq\infty$.
Thus $C_\mu$ extends uniquely to a continuous linear operator on
$L^p[0,1]$ for $1\leq p<\infty$ and the norm of this operator does
not exceed $\|\mu\|$. The same holds for $L^\infty[0,1]$: the
obstacle of $C_0[0,1]$ being non-dense in $L^\infty[0,1]$ can be
easily overcome by either using the density of $C_0[0,1]$ in
$L^\infty[0,1]$ in $*$-weak topology and $*$-weak continuity of
$C_\mu$ or by simply restricting to the non-closed invariant
subspace $L^\infty[0,1]$ of the extension of $C_\mu$ to $L^1[0,1]$.
This allows to treat each $C_\mu$ as a member of each $L(L^p[0,1])$.
From the basic properties of convolutions it also follows that the
set
$$
\A=\{C_\mu:\mu\in{\bf M}\}
$$
of truncated convolution operators is a commutative subalgebra of
$L(C_0[0,1])$ and of each $L(L^p[0,1])$. For instance, $C_\mu
C_\nu=C_\eta$, where $\eta$ is the restriction to $[0,1)$ of the
convolution of $\mu$ and $\nu$. Since $V=C_\lambda$ with $\lambda$
being the Lebesgue measure on $[0,1)$, $\A$ consists of operators
commuting with $V$. It is worth noting \cite{ms2} that on
$L^1[0,1]$, $C_0[0,1]$ and on $L^\infty[0,1]$ there are no other
continuous linear operators commuting with $V$, while this fails for
$L^p[0,1]$ with $1<p<\infty$.

In \cite{ms1,leon} it is shown that $V$ is not weakly supercyclic
(=non-supercyclic on $L^p[0,1]$ carrying the weak topology). In
\cite{mbs,leon,eve} it is demonstrated that certain truncated
convolution operators are not weakly supercyclic. L\'eon and
Piqueras \cite{leon} raised a question whether any $T\in
L(L^p[0,1])$ commuting with $V$ is not weakly supercyclic. This
question was answered affirmatively in \cite{72}. Still there
remained a possibility of existence of a hypercyclic or at least
supercyclic tuple of truncated convolution operators.

\begin{theorem} \label{mai1} Let $T_1,\dots,T_n\in\A$. Then for any
$f\in L^1[0,1]$, the projective orbit
$$
\{wT_1^{k_1}\dots T_n^{k_n}f:k_j\in\Z_+,\ w\in\K\}
$$
is nowhere dense in $L^1[0,1]$ equipped with the weak topology.
\end{theorem}

The usual comparing the topologies argument provides the following
corollary.

\begin{corollary} \label{mai11} There are no tuples of truncated
convolution operators weakly supercyclic when acting on $L^p[0,1]$
with $1\leq p<\infty$.
\end{corollary}

Since only truncated convolution operators commute with $V$ acting
on $L^1[0,1]$, the following result holds.

\begin{corollary} \label{mai12} There are no weakly supercyclic tuples of
operators on $L^1[0,1]$ commuting with $V$.
\end{corollary}

Our method applies not only to finitely generated semigroups. For
example, it also takes care of the semigroup of the
Riemann--Liouville operators, which form a subsemigroup of the
truncated convolution operators. Namely,
$$
V^zf(x)=\frac1{\Gamma(z)}\int_0^x f(t) (x-t)^{z-1}\,dt\ \ \text{with
$z\in \Pi=\{z\in\C:\re z>0\}$},
$$
where $\Gamma$ is the Euler's gamma-function. Of course, to consider
$V^z$ with non-real $z$, we need the underlying space to be over
$\C$. Clearly, $V^z=C_{\mu_z}$ with $\mu_z$ being the absolutely
continuous measure on $[0,1)$ with the density
$a_z(x)=\frac{x^{z-1}}{\Gamma(z)}$. Since $a_z\in L^1[0,1]$ for
every $z\in\Pi$, each $V^z$ is a truncated convolution operator and
therefore belongs to $\A$. Moreover, it is easy to verify that
$V^zV^w=V^{z+w}$ for every $z,w\in\Pi$ and $V=V^1$. Thus
$\{V^z\}_{z\in\Pi}$ is a semigroup and $V^n$ is exactly the $n^{\rm
th}$ power of $V$, which justifies the notation $V^z$. The map
$z\mapsto V^z$ from $\Pi$ to $L(L^p[0,1])$ is operator norm
continuous and holomorphic. Thus $\{V^z\}_{z\in\Pi}$ is a
holomorphic operator norm continuous semigroup of operators acting
on $L^p[0,1]$. In \cite{55} it is shown that for every
$\alpha\in(0,\pi/2)$, the subsemigroup $\{V^{re^{i\theta}}:r>0,\
-\alpha<\theta<\alpha\}$ is non-supercyclic on $L^p[0,1]$ for $1\leq
p<\infty$. We prove the following stronger result.

\begin{proposition}\label{sd} For every $f\in L^1[0,1]$, the set
$\{wV^zf:z\in\Pi,\ w\in\C\}$ is nowhere dense in $L^1[0,1]$ with
respect to the weak topology. In particular, the semigroup
$\{V^z\}_{z\in\Pi}$ is not weakly supercyclic.
\end{proposition}

In order to compensate for the lack of chaotic behaviour of the
orbits of operators commuting with $V$ in terms of the density in
the underlying space, we show that these operators can exhibit
chaotic behaviour in terms of the norms of the members of the orbit.
The following definition is due to Beauzamy \cite{irre}. Let $X$ be
a Banach space and $x\in X$. We say that $x$ is an {\it irregular
vector} for $T\in L(X)$ if $\ilim\limits_{n\to\infty}\|T^nx\|=0$ and
$\slim\limits_{n\to\infty}\|T^nx\|=\infty$. The concept of
irregularity was studied by Prajitura \cite{pra1}. It is worth
noting that Smith \cite{smith} constructed a non-hypercyclic
continuous linear operator $T$ on a separable Hilbert space such
that each non-zero vector is irregular for $T$.

\begin{theorem}\label{main2} There are a truncated convolution operator
$T$ and $f\in C_0[0,1]$ such that
$$
\text{$\ilim\limits_{n\to\infty}\|T^nf\|_\infty=0$\ \ and\ \
$\slim\limits_{n\to\infty}\|T^nf\|_1=\infty$.}
$$
In particular, $f$ is an irregular vector for $T$ acting on
$L^p[0,1]$ for each $p\in[1,\infty]$.
\end{theorem}

\section{Obstacles to weak supercyclicity}

In this section we develop techniques for the proof of
Theorem~\ref{mai1} and prove Proposition~\ref{sd}. We say that a
topological vector space $X$ carries a {\it weak topology} if the
topology of $X$ is the weakest topology making each $f\in Y$
continuous, where $Y$ is a fixed linear space of linear functionals
on $X$ separating the points of $X$. Of course, any weak topology is
locally convex. We say that a subset $A$ of a topological space $X$
is {\it somewhere dense} if it is not nowhere dense.

The following lemma exhibits a feature of weak topologies. Its
conclusion fails, for example, for infinite dimensional Banach
spaces.

\begin{lemma}\label{weak} Let $X$ and $Y$ be topological vector
spaces with weak topology and $A:X\to Y$ be a continuous linear
operator with dense range. Then $A(M)$ is somewhere dense in $Y$ for
every $M$ somewhere dense in $X$.
\end{lemma}

\begin{proof} Since $A$ is continuous, it is enough to show that
$A(U)$ is somewhere dense in $Y$ for every non-empty open subset $U$
of $X$. Since $A$ is linear and translation maps on a topological
vector space are homeomorphisms, it suffices to verify that $A(U)$
is somewhere dense in $Y$ for every neighborhood $U$ of $0$ in $X$.
It is easy to see that the sets of the shape
$$
U=\{x\in X:|A^*g_j(x)|<1,\ |f_k(x)|<1\ \ \text{for}\ 1\leq j\leq m,\
1\leq k\leq n\}
$$
form a basis of neighborhood of $0$ in $X$, where $g_1,\dots,g_m$
are linearly independent functionals in $Y^*$ and $f_1,\dots,f_n\in
X^*$ are such that $f_k+A^*(Y^*)$ are linearly independent in
$X^*/A^*(Y^*)$. Note that $A^*$ is injective since $A$ has dense
range and therefore the functionals $A^*g_j$ are also linearly
independent. Thus it suffices to show that $A(U)$ is somewhere dense
in $Y$ for $U$ defined in the above display. Clearly,
$$
\begin{array}{l}
A(U)=W\cap V,\ \text{where}\ W=\{y\in Y:|g_j(y)|<1\ \text{for}\
1\leq j\leq m\}\\
\qquad\qquad\qquad\qquad\text{and}\ \ V=\{Ax:|f_k(x)|<1\ \text{for}\
1\leq k\leq n\}. \
\end{array}
$$
Since $W$ is a non-empty open subset of $X$, the job will be done if
we verify that $V$ is dense in $Y$. Assume the contrary. Since $V$
is convex and balanced, the Hahn--Banach theorem implies that there
is a non-zero $f\in Y^*$ such that $|f(y)|<1$ for each $y\in V$.
That is, $|f(Ax)|=|A^*f(x)|<1$ whenever $|f_k(x)|<1$ for $1\leq
k\leq n$. It follows that $A^*f$ is a linear combination of $f_k$.
Since $A^*$ is injective, $A^*f\neq 0$ and therefore a non-trivial
linear combination of $f_k$ belongs to $A^*(Y^*)$. We have arrived
to a contradiction, which completes the proof.
\end{proof}

Recall that a subset $B$ of a vector space $X$ is called {\it
balanced} if $\lambda x\in B$ for every $x\in B$ and $\lambda\in\K$
such that $|\lambda|\leq 1$.

\begin{lemma}\label{tv1} Let $K$ be a compact subset of an infinite
dimensional topological vector space and $X$ such that $0\notin K$.
Then $\Lambda=\{\lambda x:\lambda\in\K,\ x\in K\}$ is a closed
nowhere dense subset of $X$.
\end{lemma}

\begin{proof} Closeness of $\Lambda$ in $X$ is a straightforward
exercise. Assume that $\Lambda$ is not nowhere dense. Since
$\Lambda$ is closed, its interior $L$ is non-empty. Since $K$ is
closed and $0\notin K$, we can find a non-empty balanced open set
$U$ such that $U\cap K=\varnothing$. Clearly $\lambda x\in L$
whenever $x\in L$ and $\lambda\in\K$, $\lambda\neq 0$. Since $U$ is
open and balanced the latter property of $L$ implies that the open
set $W=L\cap U$ is non-empty. Taking into account the definition of
$\Lambda$, the inclusion $L\subseteq \Lambda$, the equality $U\cap
K=\varnothing$ and the fact that $U$ is balanced, we see that every
$x\in W$ can be written as $x=\lambda y$, where $y\in K$ and
$\lambda\in{\mathbb D}=\{z\in \K:|z|\leq 1\}$. Since both $K$ and
${\mathbb D}$ are compact, $Q=\{\lambda y:\lambda\in {\mathbb D},\
y\in K\}$ is a compact subset of $X$. Since $W\subseteq Q$, $W$ is a
non-empty open set with compact closure. Since such a set exists
\cite{shifer} only if $X$ is finite dimensional, the proof is
complete.
\end{proof}

Now we can prove Proposition~\ref{sd}. Its proof resembles the proof
of the main result in \cite{72} and gives an idea of the proof of
Theorem~\ref{mai1} in the following sections. For $f\in L^1[0,1]$,
we say that the {\it infimum of the support of $f$ is $0$} if for
every $\epsilon>0$, $f$ does not vanish (almost everywhere) on
$[0,\epsilon]$.

\begin{lemma}\label{trco} Let $f,g\in L^1[0,1]$ be such that the
infima of the supports of $f$ and $g$ are $0$. Then there exist
truncated convolution operators $C,B\in L^1[0,1]$ injective and with
dense range such that $Cf=Bg$.
\end{lemma}

\begin{proof} Let $\mu$ and $\nu$ be the absolutely continuous
measures on $[0,1]$ with the densities $g$ and $f$ respectively.
Applying the  Titchmarsh theorem on the supports of convolutions to
$\mu*\nu$, we see that $C_\mu$, $C_\nu$ and their duals are
injective. Thus $C_\mu$ and $C_\nu$ are both injective and have
dense ranges. Next, $C_\mu f$ and $C_\nu g$ both equal to the
restriction to $[0,1]$ of the convolution $f*g$. Thus $C_\mu f=C_\nu
g$ and therefore $C=C_\mu$ and $B=C_\nu$ satisfy all required
conditions.
\end{proof}

\begin{proof}[Proof of Proposition~\ref{sd}] Let $f\in L^1[0,1]$. If
$f$ vanishes (almost everywhere) on $[0,\epsilon]$ for some
$\epsilon\in(0,1)$, then each $V^zf$ belongs to the space $L$ of
$g\in L^1[a,b]$ vanishing on $[0,\epsilon]$. Since $L$, being a
proper closed linear subspace of $L^1[0,1]$ is nowhere dense (in the
weak topology), the result follows. It remains to consider the case
when the infimum of the support of $f$ is $0$. Consider the
multiplication operator $M\in L(L^1[0,1])$, $Mh(x)=xh(x)$. It is
straightforward to verify that
\begin{equation}\label{vzv}
V^zM-MV^z=-zV^{z+1}\ \ \text{for every $z\in\Pi$}.
\end{equation}
Clearly, the infimum of the support of $Mf$ is also $0$. By
Lemma~\ref{trco}, there exist truncated convolution operators
$B,C\in L(L^1[0,1])$ injective and with dense range such that
$CMf=Bf$. Assume the contrary. That is, the set
$\Omega=\{wV^zf:z\in\Pi,\ w\in\C\}$ is somewhere dense in $L^1[0,1]$
carrying the weak topology. By Lemma~\ref{weak},
$V(\Omega)=\{wV^{z+1}f:z\in\Pi,\ w\in\C\}$ is also somewhere dense
in $L^1[0,1]$ with weak topology. Applying (\ref{vzv}) with $z$
replaced by $z+1$ to $f$ and multiplying by $C$ from the left, we
get $CV^{z+1}Mf-CMV^{z+1}f=-(z+1)CV^{z+2}f$. Using the commutativity
of $\A$, we obtain $V^{z+1}CMf-CMV^{z+1}f=-(z+1)CVV^{z+1}f$. Since
$CMf=Bf$, we have
$$
V^{z+1}Bf-CMV^{z+1}f=-(z+1)CVV^{z+1}f.
$$
Using commutativity of $\A$ once again, we arrive to
\begin{equation}\label{brbrb}
(CM-B)V^{z+1}f=(z+1)(CV)V^{z+1}f\ \ \text{whenever $\re z>-1$}.
\end{equation}
Pick any non-zero $g\in L^1[0,1]$, which lies in the interior of the
closure of $\{wV^{z+1}f:z\in\Pi,\ w\in\C\}$ in the weak topology.
Since $CV$ is injective, $CVg\neq 0$ and we can pick $\phi\in
(L^1[0,1])^*=L^\infty[0,1]$ such that $\phi(CVg)=(CV)^*\phi (g)\neq
0$. Take $c>0$ such that $|(CM-B)^*\phi (g)|<c|(CV)^*\phi (g)|$ and
consider the weakly open set
$$
W=\{h\in L^1[0,1]:|(CM-B)^*\phi (h)|<c|(CV)^*\phi (h)|\}.
$$
By Lemma~\ref{tv1}, the set $\{wV^{z+1}f:\re z\geq 0,\ |z|\leq c\}$
is nowhere dense in $L^1[0,1]$ with the weak topology. Since $g\in
W$, $g$ lies in the interior of the closure of
$\{wV^{z+1}f:z\in\Pi,\ w\in\C\}$ in the weak topology, we can find
$w\in\C\setminus\{0\}$ and $z\in \Pi$ such that $|z|>c$ and
$wV^{z+1}f\in W$. Using \ref{brbrb}, we have
$$
(CM-B)^*\phi(wV^{z+1}f)=(z+1)(CV)^*\phi(wV^{z+1}f).
$$
Since $wV^{z+1}f\in W$, we have
$$
|(CM-B)^*\phi(wV^{z+1}f)|<c|(CV)^*\phi(wV^{z+1}f)|.
$$
By the last two displays $|z|\leq |z+1|<c$ and we have arrived to a
contradiction.
\end{proof}

The proof of Theorem~\ref{mai1} goes along the same lines as the
proof of Proposition~\ref{sd}. However we need some extra
preparation. A {\it strongly continuous operator semigroup}
$\{T^{[t]}\}_{t\in G}$ on a topological vector space $X$ is a
collection of continuous linear operators $T^{[t]}$ on $X$ labelled
by the elements of an additive subsemigroup $G$ of $\R^n$ containing
$0$ and such that $T^{[0]}=I$, $T^{[t+s]}=T^{[t]}T^{[s]}$ for any
$t,s\in G$ and the map $t\mapsto T^{[t]}x$  from $G$ to $X$ is
continuous for each $x\in X$, where $G$ carries the topology
inherited from $\R^n$. If $n=k+m$ and $G=\R_+^k\times \Z_+^m$, then
for the sake of brevity, we shall call a strongly continuous
operator semigroup $\{T^{[t]}\}_{t\in G}$, an {operator
$(k,m)$-semigroup} on $X$. In this case we will often write $T_j$
with $1\leq j\leq n$ instead of $T^{[e_j]}$, where $e_j$ is the
$j^{\rm th}$ basic vector in $\R^n$ and we shall write $T_j^s$
instead of $T^{[se_j]}$. In this notation $T^{[t]}=T_1^{t_1}\dots
T_n^{t_n}$.

\begin{lemma}\label{gene01} Let $X$ be an infinite dimensional
topological vector space, $x\in X$, $c>0$ and $\{T^{[t]}\}_{t\in
\R_+^k\times \Z_+^m}$ be an operator $(k,m)$-semigroup on $X$. Then
the set
$$
\Omega_c=\{wT^{[t]}x:w\in\K,\ t_j\leq c\ \ \text{for}\ \ 1\leq j\leq
n\}
$$
is nowhere dense in $X$.
\end{lemma}

\begin{proof} First, observe that the general case is easily reduced
to the case $m=0$. Indeed, it follows from the fact that the union
of finitely many nowhere dense sets is nowhere dense. Thus we can
assume that $m=0$. If $x$ is not a cyclic vector for
$\{T^{[t]}\}_{t\in \R_+^k}$, $\Omega_c$ is contained in a proper
closed linear subspace of $X$ and therefore is nowhere dense. Thus
we can assume that $x$ is cyclic for $\{T^{[t]}\}_{t\in \R_+^k}$.
Without loss of generality, we can also assume that there is
$l\in\{1,\dots,k\}$ such that $T_j(X)$ is dense in $X$ if $j\geq l$
and $T_j(X)$ is not dense in $X$ if $j<l$.

\medskip

{\bf Claim 1.} \ For every $s\in\R_+^k$, $T^{[s]}x=0$ if and only if
$T^{[s]}=0$.

\begin{proof} Assume the contrary. Then there is $s\in\R_+^k$ such
that $T^{[s]}\neq 0$ and $T^{[s]}x=0$. Then $x\in L=\ker T^{[s]}\neq
X$. Since the linear space $L$ is invariant for every $T^{[t]}$ and
contains $x$, the linear span of the orbit of $x$ with respect to
$\{T^{[t]}\}_{t\in \R_+^k}$ is contained $L$. Since the latter is a
proper closed linear subspace of $X$, we have arrived to a
contradiction with the cyclicity of $x$ for $\{T^{[t]}\}_{t\in
\R_+^k}$.
\end{proof}

\medskip

{\bf Claim 2.} \ For every $s\in\R_+^k$, $T^{[s]}=0$ if and only if
$T_1^{s_1}\dots T_{l-1}^{s_{l-1}}=0$ (if $l=1$, we have the empty
product, which is always assumed to be $I$).

\begin{proof} Since $T_j(X)$ is dense in $X$ for $j\geq l$, $B(X)$ is dense in $X$, where
$B=T_l^{s_l}\dots T_k^{s_k}$. Since $T^{[s]}=AB$ with
$A=T_1^{s_1}\dots T_{l-1}^{s_{l-1}}$ and $AB=BA$, the density of the
range of $B$ implies that $T^{[s]}=0$ if and only if $A=0$.
\end{proof}

Since $x\neq 0$ and $\{T^{[t]}\}_{t\in \R_+^k}$ is strongly
continuous, we can pick $\epsilon\in (0,c)$ such that
$T_1^{\epsilon}\dots T_{l-1}^{\epsilon}x\neq 0$. By Claims~1 and~2,
$T^{[t]}x\neq 0$ whenever $t_j\leq \epsilon$ for $j<l$. Thus the
compact set
$$
K=\{T^{[t]}x:t_j\leq \epsilon\ \text{if}\ j<l\ \text{and}\ t_j\leq
c\ \text{if}\ j\geq l\}
$$
does not contain $0$. By Lemma~\ref{tv1},
$$
\Omega=\{wT^{[t]}x:w\in\K,\ t_j\leq \epsilon\ \text{if}\ j<l\
\text{and}\ t_j\leq c\ \text{if}\ j\geq l\}
$$
is closed and nowhere dense in $X$. On the other hand,
$$
\Omega_c\setminus\Omega\subseteq \bigcup_{j<l}T_j^\epsilon(X)
$$
and therefore $\Omega_c\setminus\Omega$ is nowhere dense in $X$
since $\overline{T_j^{\epsilon}(X)}\neq X$ for $j<l$. Hence
$\Omega_c$ is nowhere dense as the union of the nowhere dense sets
$\Omega$ and $\Omega_c\setminus\Omega$.
\end{proof}

\noindent {\bf Remark.} \ In the above proof we have repeatedly used
the elementary fact that if $\{T^t\}_{t\geq 0}$ is a strongly
continuous operator semigroup then $T^t$ for $t>0$ either all have
dense ranges or all have non-dense ranges.

\medskip

\begin{lemma}\label{ml} Let $X$ be an infinite dimensional
topological vector space carrying a weak topology, $\B$ be a
commutative subalgebra of $L(X)$, $x\in X$,
$\{T^{[t]}\}_{t\in\R_+^k\times \Z_+^m}$ be an operator
$(k,m)$-semigroup on $X$ such that each $T^{[t]}$ has dense range
and belongs to $\B$, $M\in L(X)$, $B,C\in \B$, $[T_j,M]=S_j\in\B$
for $1\leq j\leq n=k+m$, $CMx=Bx$, $C(X)$ is dense in $X$ and the
convex span of the operators $R_1,\dots,R_n$ does not contain the
zero operator, where
\begin{equation}\label{rj}
R_j=T_1\dots T_{j-1}S_jT_{j+1}\dots T_n.
\end{equation}
Then $O=\{wT^{[t]}x:w\in\K,\ t\in\R_+^k\times\Z_+^m\}$ is nowhere
dense in $X$.
\end{lemma}

\begin{proof} Observe that
$$
\text{$[AB,M]=B[A,M]+A[B,M]$ if $A,B,[A,M],[B,M]\in\B$.}
$$
It follows that $[T_j^r,M]=rT_j^{r-1}S_j$ whenever $r\in\N$ and
$j>k$. Similarly, $[T_j^r,M]=rT_j^{r-1}S_j$ whenever $r\geq 1$ is
rational and $j\leq k$. By strong continuity,
$[T_j^r,M]=rT_j^{r-1}S_j$ whenever $r\geq 1$ is real and $j\leq k$.
Applying the above display once again, we arrive to
$$
[T^{[t+{\bf 1}]},M]=((t_1+1)R_1+{\dots}+(t_n+1)R_n)T^{[t]}\ \
\text{for every}\ \ t\in\R_+^k\times \Z_+^m.
$$
where $R_j$ are defined in (\ref{rj}) and ${\bf 1}=(1,\dots,1)$. For
$t\in\R_+^k\times \Z_+^m$, let $N(t)=n+t_1+{\dots}+t_n$ and
$\lambda(t)=\bigl(\frac{t_1+1}{N(t)},\dots,\frac{t_n+1}{N(t)}\bigr)\in\R^n$
and for $\lambda\in\R^n$ let
$R_{[\lambda]}=\lambda_1R_1+{\dots}+\lambda_n R_n$. Then for every
$t\in\R_+^k\times\Z_+^m$, $R_{[\lambda(t)]}$ is a convex combination
of $R_j$. In this notation, the above display can be rewritten as
$$
[T^{[t+{\bf 1}]},M]=N(t)R_{[\lambda(t)]}T^{[t]}.
$$
Multiplying the equality in the above display by $C$ from the left
and applying the result to $x$, we obtain $CT^{[t+{\bf
1}]}Mx-CMT^{[t+{\bf 1}]}x=N(t)CR_{[\lambda(t)]}T^{[t]}x$. Since $C$
commutes with each $T^{[s]}$, we get
$$
T^{[t+{\bf 1}]}CMx-CMT^{[t+{\bf 1}]}x=N(t)CR_{[\lambda(t)]}T^{[t]}x.
$$
Since $CMx=Bx$ and $B$ commutes with each $T^{[s]}$, we arrive to
\begin{equation}\label{dtr}
DT^{[t]}x=N(t)CR_{[\lambda(t)]}T^{[t]}x\ \ \text{for each
$t\in\R_+^k\times\Z_+^m$, where $D=(B-CM)T^{[{\bf 1}]}$.}
\end{equation}

Assume the contrary. That is, the interior $W$ of the closure of $O$
in $X$ is non-empty. From the definitions of $O$ and $W$ it follows
that there is $s\in \R_+^k\times \Z_+^m$ such that $T^{[s]}x\in W$.
Next, we observe that the convex span $K$ of the vectors
$CT^{[s]}R_1x,\dots,CT^{[s]}R_1x$ does not contain $0$. Indeed,
assume that it is not the case. Then there are
$\lambda_1,\dots,\lambda_n\in \R_+$ such that
$\lambda_1+{\dots}+\lambda_n=1$ and $CT^{[s]}R_{[\lambda]}x=0$.
Since $0$ is not in the convex span of $R_j$, $R_{[\lambda]}\neq 0$.
Since $C$ and $T^{[s]}$ have dense ranges and commute with
$R_{[\lambda]}$, $A=CT^{[s]}R_{[\lambda]}\neq 0$. Since $A$ commutes
with each $T^{[t]}$, $\ker A$ is invariant for each $T^{[t]}$. Since
$x\in \ker A$, we have $O\subseteq\ker A$ and therefore $O$ is
nowhere dense in $X$ because $\ker A$ is a proper closed subspace of
$X$. Thus $0$ does not belong to the convex compact set $K$. By the
Hahn--Banach theorem \cite{shifer}, there is $f\in X^*$ such that
$\re f(y)>1\ \ \text{for every}\ \ y\in K$. In particular,
$$
\re f(CT^{[s]}R_jx)=\re C^*R_j^*f(T^{[s]}x)>1\ \ \text{for}\ \ 1\leq
j\leq n.
$$
Let $c=|D^*f(T^{[s]}x)|+1$. Then the open set
$$
U=\{v\in X:|D^*f(v)|<c,\ \re R_j^*C^*f(v)>1\ \text{for}\ 1\leq j\leq
n\}
$$
contains $T^{[s]}x$ and therefore $U\cap W$ is non-empty. By
Lemma~\ref{gene01}, the set $O_c=\{wT^{[t]}x:w\in\K,\ N(t)\leq c\}$
is nowhere dense in $X$. Since $O$ is dense in $U\cap W$ and $O_c$
is nowhere dense, $O\setminus O_c$ intersects $U\cap W$. Thus we can
pick $z\in\K$ and $t\in \R_+^k\times \Z_+^m$ such that $N(t)>c$ and
$u=zT^{[t]}x\in U\cap W$. Applying $f$ to the both sides of
(\ref{dtr}), we obtain $D^*f(u)=N(t)R^*_{[\lambda(t)]}C^*f(u)$.
Hence
$$
N(t)\re R^*_{[\lambda(t)]}C^*f(u)=\re D^*f(u)\leq |D^*f(u)|.
$$
Since the real number $\re R^*_{[\lambda(t)]}C^*f(u)$ is in the
convex span of the numbers $\re R_j^*C^*f(u)$, each of which is in
$(1,\infty)$ (because $u\in U$), we have $\re
R^*_{[\lambda(t)]}C^*f(u)>1$. The inclusion $u\in U$ also implies
that $|D^*f(u)|<c$. Thus by the above display, $N(t)<c$, which is a
contradiction.
\end{proof}

In order to apply Lemma~\ref{ml} to prove Theorem~\ref{mai1}, we
need more information on truncated convolution operators.

\section{Elementary properties of truncated convolution operators}

Throughout this section, when speaking of $C_\mu$, we assume that it
acts on $C_0[0,1]$ or on $L^p[0,1]$ with $1\leq p<\infty$.

First, observe that $C_\mu=I$ precisely when $\mu=\delta$, where
$\delta$ is the point mass at $0$: $\delta(\{0\})=1$ and
$\delta(A)=0$ if $0\notin A$. As we have already mentioned, the
Titchmarsh theorem on supports of convolutions implies that $C_\mu$
and $C_\mu^*$ are injective if $\inf\supp(\mu)=0$. Hence $C_\mu$ has
dense range if $\inf\supp(\mu)=0$. In the case $\inf\supp\mu=a>0$,
the same theorem ensures that $C_\mu$ is nilpotent with the order of
nilpotency being the first natural number $n$ for which $na\geq 1$.
If $\mu(\{0\})=0$, then $\mu$ is the variation norm limit of its
restrictions $\mu_n$ to $[2^{-n},1]$. Hence $C_\mu$ is the operator
norm limit of the sequence $C_{\mu_n}$ of nilpotent operators. Thus
$C_\mu$  is quasinilpotent if $\mu(\{0\})=0$. It immediately follows
that the spectrum $\sigma(C_\mu)$ is the singleton $\{\mu(\{0\})\}$
for each $\mu\in\M$. Recall that a power $T^n$ of an operator $T$ is
the identity $I$ if and only if $T$ is the direct sum of operators
of the shape $cI$ with $c^n=1$. In the case when the spectrum of $T$
is a singleton, this means that $T=cI$ with $c^n=1$. The above
observations are summarized in the following proposition.

\begin{proposition}\label{ele1} Let $\mu\in\M$. Then
\begin{itemize}
\itemsep-2pt
\item[{\rm (\ref{ele1}.1)}]$C_\mu$ is injective if and only if
$C_\mu$ has dense range if and only if $\inf\supp(\mu)=0;$
\item[{\rm (\ref{ele1}.2)}]$C_\mu^n=aI$ if and only if $C_\mu=cI$
$($equivalently, $\mu=c\delta)$ with $c^n=a;$
\item[{\rm (\ref{ele1}.3)}]$\sigma(C_\mu)=\{\mu(\{0\})\}$.
\end{itemize}
\end{proposition}

We need some extra information on truncated convolution operators.

\begin{lemma}\label{lele1} $T\in\A$ is invertible if and only if
$T=ce^A$ with $c\in\K\setminus\{0\}$ and $A\in\A$ being
quasinilpotent.
\end{lemma}

\begin{proof} Of course $ce^A$ belongs to $\A$ and is invertible if
$c\in\K\setminus\{0\}$ and $A\in\A$. Assume now that $T\in\A$ is
invertible. By (\ref{ele1}.3), $T=C_\mu$ with $c=\mu(\{0\})\neq 0$.
Thus $\mu=c\delta+\nu$, where $\nu\in\M$ and $\nu(\{0\})=0$. That
is, $T=c(I+S)$, where $S=\frac 1cC_\nu\in\A$ is quasinilpotent.
Since $S$ is quasinilpotent, the operator
$$
A=\ln(I+S)=\sum_{n=1}^\infty \frac{(-1)^{n-1}}{n}S^n
$$
is well-defined, belongs to $\A$ and is also quasinilpotent. It
remains to observe that $T=ce^A$.
\end{proof}

\begin{lemma}\label{lele2} Let $\{T^t\}_{t\geq 0}$ be a strongly
continuous operator semigroup such that each $T^t$ belongs to $\A$
and $T^1$ is invertible. Then there are a quasinilpotent $A\in\A$
and $a\in\K\setminus\{0\}$ such that $T^t=e^{ta}e^{tA}$ for
$t\in\R_+$.
\end{lemma}

\begin{proof} By Lemma~\ref{lele1}, there are $c\in\K\setminus\{0\}$
and $A\in\A$ quasinilpotent such that $T^1=ce^A$. Then for every
$k,m\in\N$, $(T^{k/m}e^{-kA/m})^m=c^kI$. By (\ref{ele1}.2),
$T^{k/m}e^{-kA/m}$ is a scalar multiple of the identity. Thus
$T^te^{-tA}$ is a scalar multiple of the identity whenever
$t\in\R_+$ is rational. By strong continuity $T^te^{-tA}$ is a
scalar multiple of the identity for each $t\in\R_+$. Thus there is a
function $\alpha:\R_+\to\K\setminus\{0\}$ such that
$T^te^{-tA}=\alpha(t)I$ for $t\in\R_+$. Since $\{T_t\}$ and
$\{e^{-tA}\}$ are strongly continuous operator semigroups, whose
members commute, $\{\alpha(t)I\}$ is a strongly continuous operator
semigroup as well. Hence $\alpha$ is continuous, $\alpha(0)=1$ and
$\alpha(t+s)=\alpha(t)\alpha(s)$ for every $t,s\in\R_+$. It follows
that there is $a\in\K$ such that $\alpha(t)=e^{ta}$ for $t\in\R_+$.
Thus $T^t=e^{ta}e^{tA}$ for each $t\in\R_+$.
\end{proof}

Let now $M$ be the operator of multiplication by the argument acting
on the same space as the truncated convolution operators:
$$
Mh(x)=xh(x).
$$

\begin{lemma}\label{lele3}Let $\mu\in\M$. Then the commutator
$[C_\mu,M]$ belongs to $\A$. Moreover, $[C_\mu,M]=C_{\mu'}$, where
$\mu'\in\M$ is the measure absolutely continuous with respect to
$\mu$ with the density $\rho(x)=-x$.
\end{lemma}

\begin{proof}It is easy to verify that the set of $\mu\in\M$
satisfying $[C_\mu,M]=C_{\mu'}$ is closed in $\M$ with respect to
the weak topology $\sigma$ provided by the natural dual pairing
$(\M,C[0,1])$. Thus it is enough to prove the required equality for
$\mu$ from a $\sigma$-dense set in $\M$. As such a set we can take
the set of absolutely continuous measures with polynomial densities.
By linearity, it suffices to prove the equality $[C_\mu,M]=C_{\mu'}$
for $\mu$ being absolutely continuous with the density $d(x)=x^n$
for $n\in\Z_+$. In the latter case the required equality is an
elementary integration by parts exercise (left to the reader).
\end{proof}

Since $\mu'=0$ if and only if $\mu=c\delta$ with $c\in\K$ and
$C_\delta=I$, we arrive to the following corollary.

\begin{corollary}\label{lele4}
The equality $[C_\mu,M]=0$ holds if and only if $C_\mu=cI$ with
$c\in\K$.
\end{corollary}

The operator $M$ is needed in order to apply Lemma~\ref{ml} to prove
Theorem~\ref{mai1}. The trickiest part of such an application is due
to the fact that the condition of $0$ being not in the convex span
of the operators $R_j$ may fail for operator semigroups contained in
$\A$. The next lemma allows us to determine exactly when does this
condition fail.

\subsection{Main lemma}

Recall that for a non-zero finite Borel $\sigma$-additive complex
valued measure $\mu$ on $\R$ with compact support, its Fourier
transform
$$
\widehat\mu(z)=\int_\R e^{-itz}\,d\mu(t)
$$
is an entire function of exponential type \cite{levin} bounded on
the real axis. Moreover, the numbers $a=\inf\supp(\mu)$ and
$b=\sup\supp(\mu)$ determine the indicator function of
$\widehat\mu$. Namely,
\begin{equation}\label{indi}
h_\mu(\theta)=\slim_{r\to+\infty}\frac{\ln|\widehat\mu(re^{i\theta})|}{r}
=\left\{\begin{array}{ll}b\sin\theta&\text{if $\theta\in[0,\pi]$}\\
a\sin\theta&\text{if $\theta\in(-\pi,0)$}\end{array}\right.
\end{equation}
Moreover, by the Cartwright theorem \cite{levin}, $\widehat\mu$ is
of completely regular growth on each ray $\{re^{i\theta}:r>0\}$ with
$\theta\in(-\pi,0)\cup(0,\pi)$. That is, for every
$\theta\in(-\pi,0)\cup(0,\pi)$, there is an open set
$E_\theta\subset(0,\infty)$ such that
\begin{equation}\label{crg}
\lim_{r\to+\infty\atop r\in
E_\theta}\frac{\ln|\widehat\mu(re^{i\theta})|}{r}=h_\mu(\theta)\ \
\text{and}\ \ \lim_{r\to+\infty}\frac{\lambda([0,r]\cap
E_\theta)}{r}=1,
\end{equation}
where $\lambda$ is the Lebesgue measure on $\R$. Recall that a
subset of $\R_+$ satisfying the second equality in (\ref{crg}) is
said to have {\it density} $1$. Thus the completely regular growth
condition means that upper limit in the definition of the indicator
function turns into the limit if we restrict ourselves to $r$ from a
suitable open set of density $1$.

\begin{lemma}\label{ML} Let $\mu_1,\dots,\mu_n$ be finite Borel
$\sigma$-additive complex valued measure on $\R$ with compact
support satisfying $\inf\supp(\mu_j)=0$ for $1\leq j\leq n$. For
each $j\in\{1,\dots,n\}$, let
$$
\nu_j=\mu_1*{\dots}*\mu_{j-1}*\mu_j'*\mu_{j+1}*{\dots}*\mu_n,
$$
where $*$ is the convolution and $\mu'$ denotes the measure
absolutely continuous with respect to $\mu$ with the density
$\rho(x)=-x$. Assume also that $c_1,\dots,c_n>0$ and
$$
\text{$\inf\supp(\nu)>0$, where $\nu=c_1\nu_1+{\dots}+c_n\nu_n$}.
$$
Then $\mu_j(\{0\})\neq 0$ for $1\leq j\leq n$.
\end{lemma}

\begin{proof} Assume the contrary. Then without loss of generality
we can assume that $\mu_1(\{0\})=0$. We can also assume that $c_j>1$
for every $j$. Indeed, multiplying all $c_j$ by the same positive
number does not change anything.

Since $\mu_1(\{0\})=0$, $\mu_1$ is the variation norm limit of the
sequence $\{\mu_{1,n}\}_{n\in\N}$ of the restrictions of $\mu_1$ to
$[2^{-n},1]$. Let $\alpha\in(0,\pi/2)$ and
$A_\alpha=\{re^{i\theta}:r\geq 0,\ \theta\in[\alpha-\pi,-\alpha]\}$.
By definition of the Fourier transform, each $\widehat{\mu_{1,n}}$
converges to $0$ as $|z|\to\infty$ for $z$ from the angle
$A_\alpha$. Moreover, $\widehat{\mu_{1,n}}$ converge uniformly to
$\widehat{\mu_1}$ uniformly on $A_\alpha$. Hence
\begin{equation}\label{li}
\lim_{r\to+\infty}\widehat{\mu_1}(re^{i\theta})=0\ \ \text{for}\ \
-\pi<\theta<0.
\end{equation}
Since $\inf\supp(\mu_j)=0$ and $\inf\supp(\mu_j')\geq 0$,
\begin{equation}\label{li1}
\text{each $\widehat{\mu_j}$ and $\widehat{\mu_j'}$ is bounded on
the half-plane $\{z\in\C:{\tt Im}\,z\leq 0\}$.}
\end{equation}
For convenience of the notation, we denote $f_j=\widehat {\mu_j}$
for $1\leq j\leq n$. Differentiating the integral defining
$\widehat{\mu_j}$, we see that $i\widehat{\mu_j'}$ is the derivative
of $\widehat{\mu_j}$:
$$
\widehat{\mu_j'}=-if_j'.
$$
Since the Fourier transform of the convolution measures is the
product of their Fourier transforms, we have
\begin{equation*}
\widehat{\nu_j}=-if_1\dots f_{j-1}f_j'f_{j+1}\dots f_n.
\end{equation*}
It immediately follows that
\begin{equation}\label{NU}
\widehat\nu=-if_1\dots f_n\sum_{j=1}^nc_j \frac{f_j'}{f_j}.
\end{equation}
Since $\inf\supp(\nu)>0$, there are $a,c>0$ such that
\begin{equation}\label{NU1}
|\widehat\nu(re^{i\theta})|\leq ce^{ar\sin\theta}\ \ \text{for
$-\pi\leq\theta\leq0$ and $r\geq 0$}.
\end{equation}
Pick $\theta\in(-\pi,0)$ such that the ray
$\ell=\{re^{i\theta}:r>0\}$ is free of zeros of the entire functions
$f_j$. Then there is a connected and simply connected open set
$U\subset\C$ such that $\ell\subset U$ and $f_j$ have no zeros on
$U$. Then the multivalued holomorphic function $f_1^{c_1}\dots
f_n^{c_n}$ splits over $U$ and we can pick its holomorphic branch
$\phi:U\to\C$. Differentiating and using (\ref{NU}), we obtain
$$
\phi'=\phi \sum_{j=1}^nc_j
\frac{f_j'}{f_j}=\frac{i\phi\widehat\nu}{f_1\dots f_n}.
$$
Using the definition of $\phi$ and the above display, we have
\begin{equation}\label{mod}
|\phi'(z)|=|\widehat\nu(z)||f_1(z)|^{c_1-1}\dots |f_n(z)|^{c_n-1}\ \
\text{and}\ \ |\phi(z)|=|f_1(z)|^{c_1}\dots |f_n(z)|^{c_n}
\end{equation}
for each $z\in U$. Since $c_j>0$, (\ref{li}), (\ref{li1}) and the
second equality in (\ref{mod}) show that
\begin{equation}\label{qwqw}
|\phi(re^{i\theta})|\to 0\ \ \text{as}\ \ r\to+\infty.
\end{equation}
Since $c_j>1$, (\ref{NU1}), (\ref{li1}) and the first equality in
(\ref{mod}) imply that there is $b>0$ such that
\begin{equation}\label{qwqw1}
|\phi'(re^{i\theta})|\leq be^{ar\sin\theta}\ \ \text{for each $r\geq
0$}.
\end{equation}
According to (\ref{qwqw}) and (\ref{qwqw1}),
$$
\phi(re^{i\theta})=-e^{-i\theta}\int_{r}^\infty \phi'(\rho
e^{i\theta})\,d\rho
$$
and therefore using (\ref{qwqw1}) once again, we get
$$
|\phi(re^{i\theta})|\leq b\int_r^\infty
e^{a\rho\sin\theta}\,d\rho=ce^{ar\sin\theta}\ \ \text{for all
$r>0$,}
$$
where $c=\frac{-b}{a\sin\theta}$. Hence
\begin{equation}\label{est1}
\slim_{r\to+\infty}\frac{\ln|\phi(re^{i\theta})|}{r}\leq
a\sin\theta<0.
\end{equation}
On the other hand, by (\ref{crg}), there are open subsets
$E_1,\dots,E_n$ of $(0,\infty)$ of density $1$ such that $\frac{\ln
|f_j(re^{i\theta})|}{r}\to0$ as $r\to+\infty$, $r\in E_j$. Since
$E=E_1\cap{\dots}\cap E_n$ is also a set of density $1$, $E$ is
unbounded and $\frac{\ln |f_j(re^{i\theta})|}{r}\to0$ as
$r\to+\infty$, $r\in E$ for $1\leq j\leq n$. Since
$\ln|\phi(re^{i\theta})|=\sum\limits_{j=1}^n c_j\ln
|f_j(re^{i\theta})|$, we arrive to
$$
\lim_{r\to +\infty\atop r\in E}\frac{\ln |\phi(re^{i\theta})|}{r}=0,
$$
which contradicts (\ref{est1}). The proof is complete.
\end{proof}

\section{Proof of Theorem~\ref{mai1}}

Since the interior of the closure of a projective orbit of a tuple
of commuting continuous linear operators does not change if we
remove the operators with non-dense range from the tuple, we can
without loss of generality assume that the operators $T_j$ in
Theorem~\ref{mai1} have dense range. Thus Theorem~\ref{mai1} is a
corollary of the following more general result.

\begin{theorem}\label{mai2} Let
$\{T^{[t]}\}_{t\in\R_+^k\times\Z_+^m}$ be an operator
$(k,m)$-semigroup on $L^1[0,1]$ consisting of truncated convolution
operators with dense range. Then for any $f\in L^1[0,1]$, the
projective orbit
$$
O=\{wT^{[t]}f:w\in\K,\ t\in\R_+\times\Z_+^m\}
$$
is nowhere dense in $L^1[0,1]$ with respect to the weak topology.
\end{theorem}

\begin{proof} Assume the contrary. That is, $O$ is somewhere dense.
We can also assume that the number $k+m$ is minimal possible for
which there exists an operator $(k,m)$-semigroup on $L^1[0,1]$ of
truncated convolution operators with dense range possessing a
somewhere dense projective orbit. Next, the infimum $c$ of the
support of $f$ must equal $0$. Indeed, otherwise $O$ is nowhere
dense as a subset of the proper closed linear subspace $L$ of $g\in
L^1[0,1]$ vanishing on $[0,c]$. Let $M:L^1[0,1]\to L^1[0,1]$ be the
multiplication by the argument operator $Mh(x)=xh(x)$. Since the
infimum of the support of $Mf$ is also $0$, Lemma~\ref{trco}
provides truncated convolution operators $B$ and $C$ with dense
range such that $CMf=Bf$. By Lemmas~\ref{lele1} and~\ref{lele2}, we
can without loss of generality assume that $1\leq j\leq k$ and
$T_j^t=e^{tA_j}$ with a quasinilpotent $A_j\in\A$ for each
invertible $T_j$. In particular, each $T_j$ with $j>k$ is
non-invertible. Denote $S_j=[T_j,M]$ for $1\leq j\leq m+k$. By
Lemma~\ref{ele1}, $T_j=C_{\mu_j}$ for $1\leq j\leq m+k$ with
$\mu_j\in\M$ and $\inf\supp(\mu_j)=0$. By Lemma~\ref{lele3},
$S_j=C_{\mu'_j}$, where $\mu'$ is the measure absolutely continuous
with respect to $\mu$ with the density $\rho(x)=-x$. If the convex
span of the operators $R_1,\dots,R_{m+k}$ with
$$
R_j=T_1\dots T_{j-1}S_jT_{j+1}\dots T_{k+m}
$$
does not contain the zero operator, Lemma~\ref{ml} with $\B=\A$
guarantees that $O$ is nowhere dense. This contradiction shows that
$0$ is in the convex span of $R_j$. Then there are $1\leq
j_1<{\dots}<j_r\leq k+m$ and $c_1,\dots,c_r>0$ such that
$c_1R_{j_1}+{\dots}+c_rR_{j_r}=0$. Since each $T_j$ has dense range
the last equality and the definition of $R_j$ implies that
$$
c_1R'_1+{\dots}+c_rR'_r=0,\ \ \text{where}\ \ R'_l=T_{j_1}\dots
T_{j_{l-1}}S_{j_l}T_{j_{l+1}}\dots T_{j_r}.
$$
Since $R'_l=C_{\nu_l}$  with $\nu_l$ being the restriction to
$[0,1)$ of the convolution
$$
\nu_l=\mu_{j_1}*{\dots}*
\mu_{j_{l-1}}*\mu'_{j_l}*\mu_{j_{l+1}}*{\dots}*\mu_{j_r},
$$
the equality $c_1R'_1+{\dots}+c_rR'_r=0$ implies that the infimum of
the support of the above convolution is at least $1$. By
Lemma~\ref{ML}, $\mu_{j_l}(\{0\})\neq 0$ for $1\leq l\leq r$. By
Lemma~\ref{ele1}, each $T_{j_l}$ is invertible. Hence $1\leq j_l\leq
k$ and $T_{j_l}^t=e^{tA_{j_l}}$ for $1\leq l\leq r$ and $t\in\R_+$
with $A_{j_l}\in\A$ being quasinilpotent. Rearranging the order of
$T_j$ with $1\leq j\leq k$, if necessary, we can without loss of
generality assume that $j_l=l$ for $1\leq l\leq r$. That is,
$T_j^t=e^{tA_j}$ for $1\leq j\leq r$ with quasinilpotent $A_j\in\A$.
It is easy to verify that
$$
S_j=[T_j,M]=[e^{A_j},M]=e^{A_j}[A_j,M]=T_j[A_j,M]\ \ \text{for
$1\leq j\leq r$}.
$$
Thus the equality $c_1R'_1+{\dots}+c_rR'_r=0$ can be rewritten as
$T_1\dots T_r(c_1[A_1,M]+{\dots}+c_r[A_r,M])=0$. Since $T_j$ are
invertible for $1\leq j\leq r$, we obtain
$[c_1A_1+{\dots}+c_rA_r,M]=0$. By Corollary~\ref{lele4},
$c_1A_1+{\dots}+c_rA_r=c I$ with $c\in\K$. Since $A_j$ commute and
are quasinilpotent, $c_1A_1+{\dots}+c_rA_r$ is also quasinilpotent
and therefore $c=0$. Thus $c_1A_1+{\dots}+c_rA_r=0$ and the
$\R$-linear span of $A_1,\dots,A_r$ coincides with the $\R$-linear
span of $A_2,\dots,A_r$. Hence
\begin{align*}
&\{T^{[t]}:t\in\R_+^k\times\Z_+^m\}\subseteq {\cal M},\ \
\text{where}
\\
&{\cal M}=\{e^{\tau_1 A_2}\dots e^{\tau_{r-1}A_r}T_{r+1}^{s_1}\dots
T_{k}^{s_{k-r}}T_{k+1}^{q_1}\dots T_{k+m}^{q_m}:\tau\in\R^{r-1},\
s\in\R_+^{k-r},\ q\in\Z_+^m\}.
\end{align*}
Thus the semigroup ${\cal M}$ admits a somewhere dense projective
orbit. Since ${\cal M}$ is the union of $2^{r-1}$ subsemigroups
${\cal M}_\epsilon$ with $\epsilon\in \{-1,1\}^{r-1}$, where
$$
{\cal M}_\epsilon=\{e^{\tau_1\epsilon_1 A_2}\dots
e^{\tau_{r-1}\epsilon_{r-1}A_r}T_{r+1}^{s_1}\dots
T_{k}^{s_{k-r}}T_{k+1}^{q_1}\dots T_{k+m}^{q_m}:\tau\in\R_+^{r-1},\
s\in\R_+^{k-r},\ q\in\Z_+^m\},
$$
at least one of the semigroups ${\cal M}_\epsilon$ admits a
somewhere dense projective orbit. Since each ${\cal M}_\epsilon$ is
an operator $(k-1,m)$-semigroup of truncated convolution operators
with dense range, we have arrived to a contradiction with the
minimality of $k+m$.
\end{proof}

\section{Proof of Theorem~\ref{main2}}

Throughout this section we use the following notation. For $a\in
L^1[0,1]$, $\nu_a$ is the absolutely continuous measure on $[0,1]$
with the density $a$ and $R_a=I+C_{\nu_a}=C_{\delta+\nu_a}$. Of
course, each $R_a$ is a truncated convolution operator.

\begin{lemma}\label{asy} Both sets
\begin{align*}
A&=\{(a,f)\in L^1[0,1]\times C_0[0,1]:\|R_a^nf\|_1\to\infty\}\\
\text{and}\ \ B&=\{(a,f)\in L^1[0,1]\times
C_0[0,1]:\|R_a^nf\|_\infty\to0\}
\end{align*}
are dense in the Banach space $L^1[0,1]\times C_0[0,1]$.
\end{lemma}

First, we shall prove Theorem~\ref{main2} assuming Lemma~\ref{asy}
and we shall prove the latter afterwards.

\begin{proof}[Reduction of Theorem~$\ref{main2}$ to Lemma~$\ref{asy}$]
For $n\in\N$, let
\begin{align*}
A_n&=\bigcup_{k>n}\{(a,f)\in L^1[0,1]\times
C_0[0,1]:\|R_a^kf\|_1>n\}\\
\text{and}\ \ B_n&=\bigcup_{k>n}\{(a,f)\in L^1[0,1]\times
C_0[0,1]:\|R_a^kf\|_\infty<n^{-1}\}.
\end{align*}
Obviously, the sets $A_n$ and $B_n$ are all open. Moreover,
$A\subseteq A_n$ and $B\subseteq B_n$ for each $n\in\N$, where $A$
and $B$ are defined in Lemma~\ref{asy}. By Lemma~\ref{asy}, $A_n$
and $B_n$ are dense in $L^1[0,1]\times C_0[0,1]$ for every $n\in\N$.
By the Baire theorem, $\Omega=\bigcap\limits_{n=1}^\infty (A_n\cap
B_n)$ is a dense $G_\delta$-subset of $L^1[0,1]\times C_0[0,1]$. In
particular, $\Omega$ is non-empty and we can pick $(a,f)\in\Omega$.
By the definition of $\Omega$,
$\slim\limits_{n\to\infty}\|R_a^nf\|_1=\infty$ and
$\ilim\limits_{n\to\infty}\|R_a^nf\|_\infty=0$. Thus the truncated
convolution operator $T=R_a$ and $f\in C_0[0,1]$ satisfy all desired
conditions.
\end{proof}

The proof of Theorem~\ref{main2} will be complete if we prove
Lemma~\ref{asy}. The proof of the latter is based upon the following
two theorems proved in \cite[Theorems~1.2 and~1.3]{mbs}.

\begin{thmA} Let $r>0$, $W\in\A$ be quasinilpotent, $1\leq
p\leq\infty$, $b>0$, $-\pi\leq\alpha\leq \pi$ and $T=I+
V^r(be^{i\alpha}I+W)$, where $V^r$ is the Riemann-Liouville
operator. Then, for each non-zero $f\in L^p[0,1]$,
$$
\lim_{n\to \infty}   \frac{\ln\|T^nf\|_p}{n^{1/(r+1)}} =
(r+1)b^{1/(r+1)}\Bigl(\frac{1-\inf \supp(f)}{r}\Bigr)^{r/(r+1)}
\cos_+\Bigl(\frac{\alpha}{r+1}\Bigr),
$$
where $\cos_+(t)=\max\{\cos t,0\}$. Furthermore, the norms
$\|T^n\|_p$ of the operators $T^n$ on the Banach space $L^p[0,1]$
satisfy
$$
\lim_{n\to \infty}\frac{\ln\|T^n\|_p}{n^{1/(r+1)}}=
(r+1)b^{1/(r+1)}\cos_+\Bigl(\frac{\alpha}{r+1}\Bigr).
$$
\end{thmA}

\begin{thmB}Let $c>0$, $1\leq p\leq\infty$, $V$ be the Volterra operator
and let $X$ be the set of positive monotonically non-increasing
sequences $a=\{a_n\}_{n=0}^\infty$ such that
$\sum\limits_{n=1}^{\infty}\frac{\ln a_n}{n^{3/2}}>-\infty$. Then
for any non-zero $f\in L_p[0,1]$, there exists $a\in X$ for which
$a_n\leq \|(I-cV)^nf\|_p$ for any $n\in\N$. Conversely for any $a\in
X$ there exists a non-zero $f\in L_p[0,1]$ for which
$\|(I-cV)^nf\|_p\leq a_n$ for any $n\in\N$. Moreover if $1\leq
p<\infty$ then the set of $f\in L_p[0,1]$ for which
$\|(I-cV)^nf\|_p=O(a_n)$ is dense in $L_p[0,1]$.
\end{thmB}

\begin{proof}[Proof of Lemma~$\ref{asy}$] Take
$a(x)=1+a_1x+{\dots}+a_nx^n$ being a polynomial with the free term
$1$ and $f$ being any non-zero function from $C_0[0,1]$. Then
$s=\inf\supp(f)\in[0,1)$. It is easy to see that
$R_a=I+V+\frac{a_1}{2}V^2+{\dots}+\frac{a_n}{n!}V^{n+1}$. Hence
$R_a=I+V(I+W)$, where $W$ is a quasinilpotent operator from $\A$. By
Theorem~A with $b=r=p=1$ and $\alpha=0$,
$$
\lim_{n\to \infty}   \frac{\ln\|R_a^nf\|_1}{n^{1/2}} =
2(b(1-s))^{1/2}>0.
$$
It immediately follows that $\|R_a^nf\|_1\to\infty$. Thus $(a,f)\in
A$ for every non-zero $f\in C_0[0,1]$ and every polynomial $a$ with
the free term 1. Since the set of such polynomials is dense in
$L^1[0,1]$, $A$ is dense in $L^1[0,1]\times C_0[0,1]$.

Assume now that $a(x)=-1+a_1x+{\dots}+a_nx^n$ is a polynomial with
the free term $-1$. Then
$R_a=I-V+\frac{a_1}{2}V^2+{\dots}+\frac{a_n}{n!}V^{n+1}=(I-V)(I+zV^k(I+W))$
with $z\in\K\setminus\{0\}$, $k\geq 2$ and $W$ being a
quasinilpotent operator from $\A$. Pick $b>0$ and
$\alpha\in(-\pi,\pi]$ such that $z=be^{i\alpha}$. Then $R_a=(I-V)T$,
where $T=I+be^{i\alpha}V^k(I+W)$. By Theorem~A,
\begin{equation}\label{tn}
\lim_{n\to \infty}\frac{\ln\|T^n\|_\infty}{n^{1/(k+1)}}=
(k+1)b^{1/(k+1)}\cos\Bigl(\frac{\alpha}{k+1}\Bigr).
\end{equation}
Pick your favorite numbers $c,d$ such that $\frac13<c<d<\frac12$ and
consider the sequence $s_n=e^{-n^{d}}$ for $n\in\N$. Since
$d<\frac12$, $\sum\limits_{n=1}^{\infty}\frac{\ln
s_n}{n^{3/2}}>-\infty$ and therefore Theorem~B implies that the set
$$
M=\{g\in L^1[0,1]:\|(I-V)^ng\|_1=O(s_n)\}
$$
is a dense subset of $L^1[0,1]$. Since $V:L^1[0,1]\to C_0[0,1]$ is a
bounded linear map with dense range, $V(M)$ is a dense subset of
$C_0[0,1]$. Since for every $g\in M$, $\|(I-V)^nVg\|_\infty\leq
c\|(I-V)^ng\|_1$, where $c$ is the norm of $V$ as an operator from
$L^1[0,1]$ to $C_0[0,1]$, we see that $\|(I-V)^nf\|_\infty=O(s_n)$
for every $f\in V(M)$. Hence
$$
\|R_a^nf\|_\infty=\|T^n(I-V)^nf\|_\infty\leq
\|T^n\|_\infty\|(I-V)^nf\|_\infty=O(s_n\|T^n\|_\infty)\ \ \text{for
each $f\in V(M)$}.
$$
Since $k\geq 2$ and $c>\frac13$, from (\ref{tn}) it follows that
$\|T^n\|_\infty=O(e^{n^c})$. Since $s_n=e^{-n^d}$, by the above
display, $\|R_a^n\|_\infty=O(e^{n^c-n^d})$. Since $d>c$,
$e^{n^c-n^d}\to 0$ and therefore $\|R_a^nf\|_\infty\to 0$ for $f\in
V(M)$. Thus $(a,f)\in B$ if $f\in V(M)$ and $a$ is a polynomial with
the free term $-1$. Since the set of such polynomials is dense in
$L^1[0,1]$ and $V(M)$ is dense in $C_0[0,1]$, $B$ is dense in
$L^1[0,1]\times C_0[0,1]$.
\end{proof}

The following questions remains open.

\begin{question}\label{q1} Does there exist a truncated convolution
operator $T$ on $L^2[0,1]$ such that every non-zero $f\in L^2[0,1]$
is an irregular vector for $T$?
\end{question}

As we have mentioned, for $1<p<\infty$ there are continuous linear
operators on $L^p[0,1]$ commuting with $V$ other than truncated
convolutions. Thus the following question remains open. Although
probably a negative answer could be obtained by a not so
sophisticated modification of the proof of Theorem~\ref{mai1}.

\begin{question}\label{q2} Let $1<p<\infty$. Does there exist a weakly
supercyclic tuple of continuous linear operators on $L^p[0,1]$
commuting with $V$?

\end{question}

\vfill \break


\small\rm

\vskip1truecm

\scshape

\noindent Stanislav Shkarin

\noindent Queens's University Belfast

\noindent Pure Mathematics Research Centre

\noindent University road, Belfast, BT7 1NN, UK

\noindent E-mail address: \qquad {\tt s.shkarin@qub.ac.uk}


\begin{thebibliography}{99}

\itemsep=-2pt

\bibitem{bama-book}F.~Bayart and E.~Matheron, \it Dynamics of linear
operators, \rm Cambridge University Press, 2009

\bibitem{irre}B.~Beauzamy, \it Introduction to Operator Theory and Invariant
Subspaces, \rm North-Holland, Amsterdam, 1988

\bibitem{mbs}S.~Bermudo, A.~Montes-Rodr\'{\i}guez and S.~Shkarin, \it Orbits
of operators commuting with the Volterra operator, \rm J. Math.
Pures. Appl. \bf89\rm\ (2008), 145-173

\bibitem{eve}S.~Eveson, \it Non-supercyclicity of Volterra
convolution and related operators, \rm Integral Equations Operator
Theory \bf62\rm\ (2008), 585--589

\bibitem{feld}N.~Feldman, \it Hypercyclic tuples of operators and
somewhere dense orbits, \rm J. Math. Anal. Appl. \bf 346\rm\ (2008),
82--98

\bibitem{gm}E.~Gallardo-Guti\'errez and A.~Montes-Rodr\'{\i}guez,
\it The Volterra operator is not supercyclic, \rm Integral Equations
Operator Theory \bf50\rm\ (2004), 211--216

\bibitem{leon}F.~L\'eon-Saavedra and A.~Piqueras-Lerena, \it Cyclic properties
of Volterra Operator II\ \ \rm [preprint]

\bibitem{levin}B.~Levin, \it Distribution of zeros of entire
functions, \rm AMS, 1980

\bibitem{ms1}A.~Montes-Rodr\'{\i}guez and S.~Shkarin, \it
Non-weakly supercyclic operators, \rm J. Operator Theory \bf58\rm\
(2007), 39--62

\bibitem{ms2}A.~Montes-Rodr\'{\i}guez and S.~Shakrin, \it New results on a
classical operator, \rm Contemp. Math. \bf393\rm\ (2006), 139--158

\bibitem{pra1}G.~Prajitura, \it Irregular vectors of Hilbert space
operators,\rm\ J. Math. Anal. Appl. \bf354\rm\ (2009), 689--697

\bibitem{shifer}H.~Sch\"afer, \it Topological Vector Spaces, \rm
Macmillan, New York, 1966

\bibitem{55}S.~Shkarin, \it Antisupercyclic operators and orbits of
the Volterra operator, \rm J. Lond. Math. Soc. \bf73\rm\ (2006),
506--528

\bibitem{72}S.~Shakrin, \it Operators, commuting with the Volterra operator,
are not weakly supercyclic, \rm Integral Equations and Operator
Theory [to appear], Electronic DOI 10.1007/s00020-010-1790-y

\bibitem{smith}L.~Smith, \it A nonhypercyclic operator with orbit-density
properties, \rm Acta Sci. Math. (Szeged) \bf74\rm\ (2008), 741--754

\end{thebibliography}
\end{document}